\documentclass[11pt]{article}
\usepackage{float}
\usepackage{graphicx}
\usepackage{amsmath,amsthm}
\usepackage{amssymb}
\usepackage{mathrsfs}
\usepackage{psfrag}

\restylefloat{figure}

\newtheorem{theorem}{\bf Theorem}
\newtheorem{assumption}{\bf Assumption}

\newtheorem{lemma}[theorem]{\bf Lemma}

\newcommand{\bepsilon}{\mbox{\boldmath{$\varepsilon$}}}
\def\be{\begin{equation}}
\def\ee{\end{equation}}
\def\ben{\begin{eqnarray}}
\def\een{\end{eqnarray}}

\renewcommand{\div}{\rm div\,}
\newcommand{\bR}{\mathbb{R}}

\newcommand{\bH}{\mathbb{H}}
\newcommand{\bQ}{\mathbf{Q}}

\newcommand{\bL}{\mathbb{L}}
\newcommand{\bv}{\mathbf{v}}

\newcommand{\bw}{\mathbf{w}}

\newcommand{\bphi}{\mbox{\boldmath{$\phi$}}}
\newcommand{\bpsi}{\mbox{\boldmath{$\psi$}}}
\newcommand{\bmu}{\mbox{\boldmath{$\mu$}}}

\newcommand{\bzeta}{\mbox{\boldmath{$\zeta$}}}
\newcommand{\boeta}{\mbox{\boldmath{$\eta$}}}
\newcommand{\bPi}{\mbox{\boldmath{$\Pi$}}}

\newcommand{\CT}{\mathcal{T}}
\newcommand{\CC}{\mathcal{C}}
\newcommand{\cS}{\mathcal{S}}
\newcommand{\CE}{\mathcal{E}}
\newcommand{\CB}{\mathcal{B}}
\newcommand{\CN}{\mathcal{N}}
\newcommand{\cI}{\mathcal{I}}
\newcommand{\bx}{\mathbf{x}}

\newcommand{\bu}{\mathbf{u}}
\newcommand{\bN}{\mathbb{N}}

\newcommand{\bV}{\mathbf{V}}

\newcommand{\bM}{\mathbf{M}}

\newcommand{\tB}{\mathtt{B}}

\newcommand{\tD}{\mathtt{D}}
\newcommand{\tA}{\mathtt{A}}

\title{A  simple finite element method for 
Reissner--Mindlin plate equations using the Crouzeix-Raviart element 
and the standard linear finite element}
\author{Bishnu P.~Lamichhane
\thanks{School of Mathematical \& Physical Sciences,
Mathematics Building - V127,
University of Newcastle,
University Drive,
Callaghan, NSW 2308, Australia, 
 {\tt Bishnu.Lamichhane@newcastle.edu.au}}}
\begin{document}
\maketitle

\begin{abstract}
We present a simple finite element method for the 
discretization of Reissner--Mindlin plate equations.
The finite element method is 
based on using the nonconforming Crouzeix-Raviart finite element space 
for the transverse displacement, and the standard linear 
finite element space for the rotation of the transverse normal vector.
We also present two examples for 
the discrete Lagrange multiplier space for the proposed formulation. 

\end{abstract}

{\bf Key words}
Reissner--Mindlin plate, 
Lagrange multiplier, biorthogonal system, Crouzeix-Raviart element, 
a priori error estimates \\

{\bf AMS subject classification}.
65N30, 74K20

 \section{Introduction}
 It is a challenge to 
 design a  simple finite element scheme for  Reissner--Mindlin plate equations, 
 which does not lock when the plate thickness becomes close to zero.
 A standard discretization normally does not provide 
 a uniform convergence with respect to 
 the plate thickness.  This problem is often referred to as {\em locking}. 
 Many finite element 
 techniques are developed over the past twenty years 
 to avoid {\em locking} and obtain a uniform convergence 
 with respect to the plate thickness 
 \cite{AF89,BF91,AB93,CL95,Lov96,AF97,CS98,Bra96,FT00,Bra01,ACC02,Lo05}.
Most of these finite element methods are either too 
complicated or too expensive to implement.

In this paper, we present a very simple finite element method 
for Reissner--Mindlin plate equations providing a uniform 
convergence with respect to the plate thickness. 
We consider both {\em simply supported} and  
{\em clamped} boundary condition. 
Previously, a simple finite element method for 
Reissner--Mindlin plate equations is presented  \cite{Lam13c}
 for the case of {\em clamped} boundary condition, where 
 we have enriched 
 the standard linear finite element space with element-wise 
 bubble functions for  the approximation of the transverse displacement 
 to ensure the stability of the system. In this paper, 
 we show that the stability is ensured if we 
 use the nonconforming  Crouzeix-Raviart finite element space to 
 approximate the transverse displacement, whereas 
 other variables are discretized as in \cite{Lam13c}. That means 
 each component of the rotation of the transverse normal vector is 
 approximated by the standard linear finite element, whereas 
 we present two examples of the discrete Lagrange multiplier space.  The first 
one is based on the standard linear finite element 
space, whereas the second one is 
based on a dual Lagrange multiplier space 
proposed in \cite{BWHabil,KLP01}. 
The main advantage of 
using a dual Lagrange multiplier space 
is that it allows an efficient static 
condensation of the degrees of freedom 
associated with the Lagrange multiplier space. 
This leads to a positive-definite system. 
An iterative solver performs better for 
a positive-definite system than for a saddle point 
system. Hence the dual 
Lagrange multiplier space leads to 
a more efficient numerical scheme. 
The case of {\em clamped} boundary condition 
  is treated by using the idea of 
  mortar finite elements for the boundary 
  modification \cite{BD98, Lam13c}.

We now want to point out some links of this present work with some previously 
presented nonconforming finite element schemes for Reissner-Mindlin plate equations
 \cite{AF89,AF97,Lo05}. 
For example, the finite element scheme presented in \cite{AF89} 
uses the nonconforming Crouzeix-Raviart finite element for the transverse displacement, but 
each component of the rotation of the transverse normal vector is approximated by 
the standard linear finite element space enriched 
with element-wise bubble functions, and the Lagrange multiplier space 
is discretized by the space of piecewise constant functions. 
We do not need to use bubble functions in our formulation, and 
hence our finite element method is more efficient than 
this  finite element scheme. 

The finite element scheme in \cite{Lo05} uses 
the nonconforming Crouzeix-Raviart finite element  for 
the transverse displacement and  each component of the rotation of the transverse normal 
vector, whereas the Lagrange multiplier is approximated by 
the space of piecewise constant functions. 
Since the Crouzeix-Raviart element is 
used for each component of the rotation of the transverse normal vector, 
a stabilization is introduced in order to achieve 
Korn's inequality. Since we use 
a conforming approach for the rotation of the transverse normal vector 
we do not need the stabilization for our finite element 
scheme.

The rest of the 
paper is planned as follows. 
The next section briefly recalls the Reissner--Mindlin plate 
equations in a modified form as given in \cite{AB93}. 
We describe our finite element method and present 
assumptions on the discrete Lagrange multiplier space in 
Section \ref{sec:fe}. Section \ref{sec:ssbc} 
is devoted to the presentation of 
two examples of the discrete Lagrange multiplier space 
for the {\em simply supported boundary condition}, and 
we show the modification of the discrete Lagrange multiplier space 
for the {\em clamped boundary condition} in 
Section \ref{sec:clbc}. 
Finally, we apply static condensation of  the Lagrange multiplier in 
Section \ref{sec:pd} before drawing 
a conclusion in the last section.

\section{A mixed formulation of Reissner--Mindlin plate}
Let $\Omega\subset \bR^2$ be a bounded region with 
polygonal  boundary. 
We need the following Sobolev spaces for the variational 
formulation of the Reissner--Mindlin plate with the plate 
thickness $t$: 
\[\bH^1(\Omega) = [H^1(\Omega)]^2,\quad 
\bH^1_0(\Omega) = [H^1_0(\Omega)]^2,\quad \text{and}\quad 
\bL^2(\Omega) = [L^2(\Omega)]^2.\]
Since the regularity of the shear stress depends on 
the plate thickness $t$, we  use the Hilbert space for the 
shear stress depending on $t$. 
Let $(\bH_0^1(\Omega))'$ and $(\bH^1(\Omega))'$ 
be the dual spaces of $\bH_0^1(\Omega)$ and $\bH^1(\Omega)$, 
 respectively. Now we define the Hilbert space for the shear stress as 
\[ \bM_t:= \begin{cases}
\{\bzeta \in (\bH_0^1(\Omega))' :\, \||\bzeta\||_t  < \infty\},\quad 
\text{for the clamped boundary},\\
\{\bzeta \in (\bH^1(\Omega))' :\, \||\bzeta\||_t  < \infty\},\quad 
\text{for the simply supported boundary},
\end{cases}\]
where the norm $\||\cdot\||_t$ is defined as 
\[  \||\bzeta\||_t = \begin{cases}
\|\bzeta\|_{(\bH_0^1(\Omega))' } + \|\nabla\cdot\bzeta \|_{H^{-1}(\Omega)}  
+ t \|\bzeta \|_{L^2(\Omega)}\quad 
\text{for the clamped boundary},\\
\|\bzeta\|_{(\bH^1(\Omega))' } + \|\nabla\cdot\bzeta \|_{H^{-1}(\Omega)}  
+ t \|\bzeta \|_{L^2(\Omega)}\quad 
\text{for the simply supported boundary}.
\end{cases}
\]
In order to get a unified framework for the 
clamped and simply supported boundary  of 
Reissner-Mindlin plate we define the space $\bV$ for the transverse displacement 
as 
\[ \bV:= \begin{cases}
\bH_0^1(\Omega), \quad 
\text{for the clamped boundary},\\
\bH^1(\Omega),\quad 
\text{for the simply supported boundary}.
\end{cases}
\]
We consider the following modified mixed formulation of 
Reissner--Mindlin plate equations 
proposed in \cite{AB93}.  The mixed formulation 
is to find $(\bphi,u,\bzeta) \in 
\bV \times H^1_0(\Omega)\times \bM_t$ such that 
\begin{equation} \label{rmeqn}
\begin{array}{lccccccc}
 a(\bphi,u;\bpsi,v)&+b(\bpsi,v;\bzeta)&=&\ell(v),\quad &(\bpsi,v) 
&\in &\bV\times H^1_0(\Omega), &\\
b(\bphi,u;\boeta) &- \frac{t^2}{\lambda(1-t^2)} (\bzeta,\boeta)& =& 0, \quad 
&\boeta& \in& \bM_t,&
\end{array}
\end{equation} 
where $\lambda$ is a material constant depending on Young's modulus 
$E$ and Poisson ratio $\nu$, and 
\begin{eqnarray*}
a(\bphi,u;\bpsi,v) &=& \int_{\Omega} \CC \bepsilon(\bphi):\bepsilon(\bpsi)\,d\bx +\lambda \int_{\Omega} (\bphi-\nabla u)\cdot(\bpsi-\nabla v)\,d\bx, \\
b(\bpsi,v;\boeta)& =& \int_{\Omega} (\bpsi-\nabla v)\cdot\boeta\,d\bx,\quad 
\ell(v) = \int_{\Omega} g\,v\,d\bx.
\end{eqnarray*}
Here  $g$ is the body force, 
$u$ is the transverse displacement or normal deflection 
of the mid-plane section of $ \Omega$, $\bphi$ is the rotation of the 
transverse normal vector,  $\bzeta$ is the Lagrange 
multiplier, $\CC$ is the fourth order tensor, and 
$\bepsilon(\bphi)$ is the symmetric part of the gradient of $\bphi$. 
In fact, $\bzeta$ is the scaled shear 
stress defined by 
\[ \bzeta = \frac{\lambda (1-t^2)} {t^{2}} \left(\bphi - \nabla u\right).\]

\section{A finite element discretization}\label{sec:fe}

We consider a quasi-uniform triangulation $\CT_h$ of the 
polygonal domain $\Omega$, where $\CT_h$
consists of 
triangles where $h$ denotes the mesh-size. 
 Note that $\CT_h$ denotes the set
of elements.  For an element $T \in \CT_h$, let $P_n(T)$ be the 
set of all polynominals of degree less than or equal to $ n \in
\bN\cup \{0\}$ in $T$.
Let $\{\bx_i\}_{i=1}^N$ be the set of all vertices of 
the triangulation $\CT_h$, and $\CN_h=\{i\}_{i=1}^N$. 
Let $\CE_h$ be the set of all edges of elements in $\CT_h$, 
and $[v_h]_e$ the jump of the function $v_h$ across the edge
$e$. 

We consider a nonconforming 
finite element space $S_h$ for the 
transverse displacement, where 
the continuity of a function $v_h\in S_h$ 
across an edge  $e\in \CE_h$ 
will be enforced according to 
\[ J_e(v_h) := \int_{e}[v_h]_e \,d\sigma =0.\]
This is the standard nonconforming 
Crouzeix-Raviart finite elment space $S_h$ \cite{CR73} defined as 
\[ S_h := \{v_h \in L^2(\Omega):\, v_h|_{K} \in P_1(K),
\; 
K \in \CT_h, \; J_e(v_h) =0,\; e \in \CE_h\}.\]

The finite element basis functions of $S_h$ are associated 
with the mid-points of the edges of triangles.
To impose the homogeneous Dirichlet boundary condition on $\Gamma$ we 
 define $W_h$ as a subset of $S_h$ where 
\[ W_h : = \{ v_h \in S_h:\, \int_{e} v_h\,d\sigma =0,
e \in \CE_h \cap \Gamma\}.\]

As $W_h \not\subset H_0^1(\Omega)$, we cannot use 
the standard $H^1$-norm for an element in $W_h$. 
So we define a broken norm on $W_h$ as 
\[ \|v_h\|_{1,h} : =\sqrt{ \sum_{T \in \CT_h}
  \|v_h\|^2_{1,T}},\quad v_h \in W_h,
\]
and an element-wise defined gradient $\nabla_h$ and divergence $ \nabla_h \cdot$ as 
\[ 
\nabla_h u_h|_{T}  = \nabla (u_h|_{T}),\quad\text{and}\quad  
\nabla_h \cdot u_h|_{T}  = \nabla \cdot (u_h|_{T}),\quad\text{on}\quad 
T,\quad T \in \CT_h.\]
We note that
the standard linear finite element space 
enriched with element-wise defined bubble functions 
to approximate the transverse displacement is used in \cite{AB93,Lam13c}, 
which is a conforming approach. Our approach 
here is nonconforming for the transverse displacement 
since $W_h \not\subset H_0^1(\Omega)$.

Each component of the rotation of the 
 transverse normal vector is 
discretized by using the standard linear finite element space 
\[ K_h :=\{ q_h \in H^1(\Omega):\, 
q_h|_{K} = P_1(K), \; K \in \CT_h\},\quad 
K^0_h := K_h \cap H^1_0(\Omega).\]

The finite element space for the rotation of the 
 transverse normal vector is 
\[ \bV_h:= \begin{cases}
[K_h]^2,\quad 
\text{for the clamped boundary},\\
[K^0_h]^2, \quad 
\text{for the simply supported boundary}.
\end{cases}\]

\subsection{Discrete Lagrange multiplier spaces} 

Let $M_h\subset L^2(\Omega)$ be a piecewise 
polynomial space with respect to the mesh $\CT_h$  
used to discretize each component of the Lagrange multiplier $\bzeta \in \bM_t$. 
The discrete Lagrange multiplier space is defined as 
$\bM_h := [M_h]^2$. 
The Lagrange multiplier $\bzeta\in \bM_t$ is the shear stress, and 
the discrete space for the shear stress should have the approximation property 
in the $L^2$-norm. Hence we need
\[
\inf_{\mu_h \in M_h}\|v-\mu_h\|_{L^2(\Omega)}\leq 
Ch |v|_{1,\Omega},\quad v \in H^1(\Omega).
\]

\begin{figure}[!ht]
\begin{center}
 \includegraphics[width =0.68\textwidth]{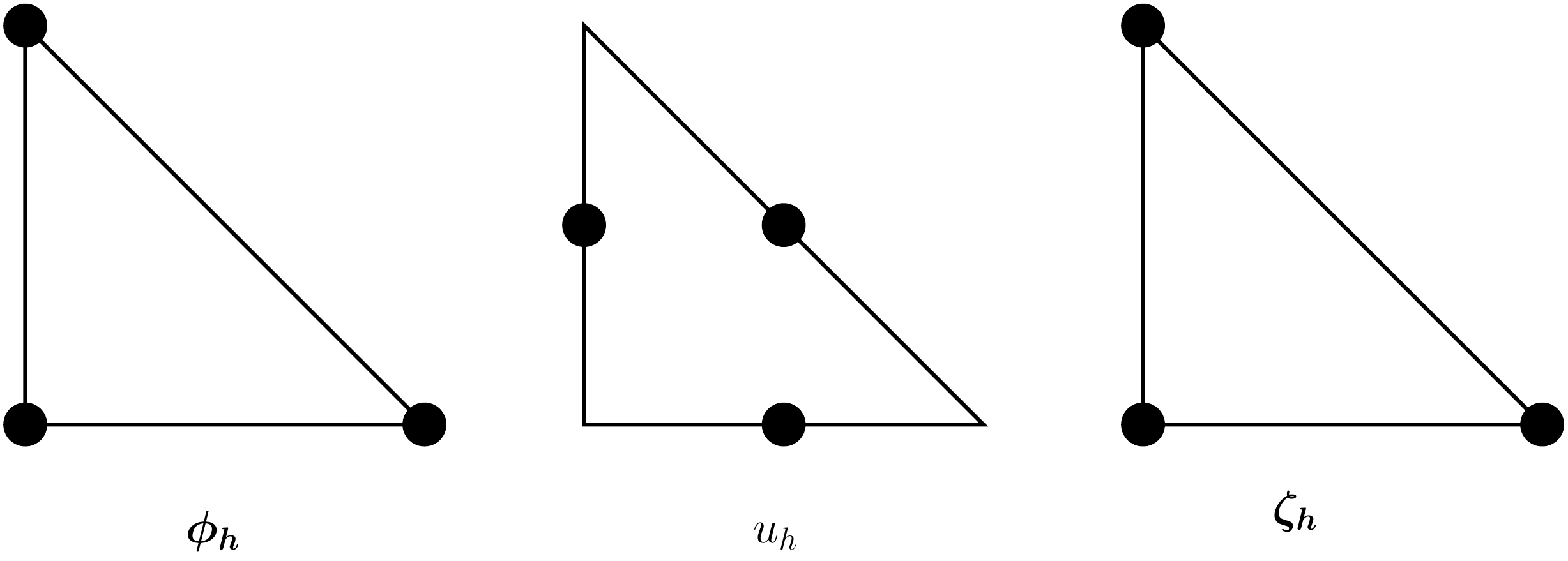}
 \caption{Degrees of freedom for the finite element spaces } 
\end{center}
 \label{SI}
\end{figure}

The finite element formulation is to 
find $(\bphi_h,u_h,\bzeta_h) \in \bV_h \times W_h \times \bM_h$ 
such that  
\begin{equation}\label{dsaddle}
\begin{array}{ccccccc}
a_h(\bphi_h,u_h;\bpsi_h,v_h)&+
b_h(\bpsi_h,v_h;\bzeta_h)&=&\ell(v_h),\quad &(\bpsi_h,v_h) 
&\in &\bV_h \times W_h, \\
b_h(\bphi_h,u_h;\boeta_h) &- \frac{t^2}{\lambda(1-t^2)} (\bzeta_h,\boeta_h)& =& 0, \quad 
&\boeta_h& \in& \bM_h,
\end{array}
\end{equation}
where 
\begin{eqnarray*}
a_h(\bphi_h,u_h;\bpsi_h,v_h) &=& \int_{\Omega} \CC \bepsilon(\bphi_h):\bepsilon(\bpsi_h)\,d\bx +
\lambda \int_{\Omega} (\bphi_h-\nabla_h u_h)\cdot(\bpsi_h-\nabla_h v)\,d\bx, \\
b_h(\bpsi_h,v_h;\boeta_h)& =& \int_{\Omega} (\bpsi_h-\nabla_h v_h)\cdot\boeta_h\,d\bx.
\end{eqnarray*}

In order to get stability and optimality of our finite element scheme we 
impose the following assumptions on the discrete Lagrange multiplier space 
as in \cite{AB93,Lam13c}.
\begin{assumption}\label{A1A2}
\begin{itemize} 
\item[\ref{A1A2}(i)] $\dim \bM_h = \dim \bV_h$.
\item[\ref{A1A2}(ii)] There is a constant $\beta>0$ independent of 
the triangulation $\CT_h$ such that 
\begin{eqnarray}
\|\bphi_h\|_{L^2(\Omega)} \leq \beta \sup_{\bmu_h \in M_h \backslash\{0\}} 
\frac{\int_{\Omega} \bmu_h\cdot\bphi_h\,d\bx} {\|\bmu_h\|_{L^2(\Omega)}},
\quad \bphi_h \in \bV_h.
\end{eqnarray}
\item[\ref{A1A2}(iii)] The space $\bM_h$ has the approximation property:
\begin{equation}
\inf_{\bmu_h \in \bM_h}\|\bmu-\bmu_h\|_{L^2(\Omega)}\leq 
Ch |\bmu|_{1,\Omega},\quad \bmu \in \bH^1(\Omega).
\end{equation}
\item[\ref{A1A2}(iv)]
There exist two bounded linear projectors  
$\bQ_h :\bH_0^1(\Omega) \rightarrow \bV_h$ 
and 
$\Pi_h : H_0^1(\Omega) \rightarrow W_h$ 
for which 
\[ b_h( Q_h\bpsi,\Pi_h v;\boeta_h) = b(\bpsi,v;\boeta_h),\quad 
\boeta_h \in \bM_h.\]
\end{itemize}
\end{assumption} 
If these assumptions are satisfied, we obtain 
an optimal error estimate for the finite element approximation, 
see \cite{AB93}.  
We immediately see that 
the bilinear forms 
$a_h(\cdot,\cdot)$, $b_h(\cdot,\cdot)$ and 
the linear form $\ell(\cdot)$ are continuous with respect to the 
spaces $\bV_h \times W_h$, $\bV_h\times W_h \times \bM_h$ and  
$W_h$, respectively, where  the broken norm 
$\|\cdot\|_{1,h}$ is used for functions in $W_h$. 
Similarly, the coercivity of the bilinear form $a_h(\cdot,\cdot)$ 
over the space $\bV_h \times W_h$ also holds 
due to the Korn's and Poincar\'e inequality. Note that 
the use of Crouzeix-Raviart element does not cause problem 
here as it is only used to discretize the transverse displacement.
Then under above assumptions we have the following theorem from 
the theory of saddle point problems \cite{BF91,Bra01,AB93}.
The proof of the following theorem is 
quite similar to the convergence result in \cite{Lo05}.
\begin{theorem}\label{th0}
Let $(\bphi,u,\bzeta) \in \bV \times W \times \bM$   be the solution 
\eqref{rmeqn} and $(\bphi_h,u_h,\bzeta_h) \in \bV_h \times W_h \times \bM_h$ 
 of \eqref{dsaddle}.  Then under Assumptions \ref{A1A2}(i)--(iv) 
 there exists a constant $C$ independent of $t$ and $h$ such that 
\begin{eqnarray*}
 \|\bphi-\bphi_h\|_{1,\Omega} + \|u - u_h \|_{1,h} 
+ \||\bzeta -\bzeta_h\||_t  \leq \\
C h \left(\|\bphi\|_{H^2(\Omega)} + \|u\|_{H^2(\Omega)}  + \|\bzeta\|_{H(\div,\Omega)}+ 
t\|\bzeta\|_{1,\Omega}\right),
\end{eqnarray*}
where, we assume that  $\bphi \in \bH^2(\Omega)$, 
$u \in H^2(\Omega)$ and $\bzeta\in \bH^1(\Omega)$.
\end{theorem}

\section{Simply supported boundary condition} \label{sec:ssbc}
We first consider the case of simply supported boundary condition. 
We consider two examples of the discrete Lagrange multiplier space 
 satisfying Assumptions \ref{A1A2}(i)--(iv). 
\subsection{First example for $\bM_h$}
Let $\bM_h:= [K_h]^2$. 
We can see that this example satisfies Assumptions \ref{A1A2}(i)--(iii). 
Now we prove that it also satisfies Assumptions \ref{A1A2}(iv). 
\begin{theorem}\label{th1}
There exist two bounded linear projectors 
$\bQ_h :\bH^1(\Omega) \rightarrow \bV_h$ 
and 
$\Pi_h : H_0^1(\Omega) \rightarrow W_h$ 
for which 
\begin{equation}\label{op} 
b( \bQ_h\bpsi,\Pi_h v;\boeta_h) = b(\bpsi,v;\boeta_h),\quad 
\boeta_h \in \bM_h.\end{equation}
\end{theorem}
In order to prove Theorem \ref{th1} we use the following result 
proved in \cite{Lam13d}. 
\begin{lemma}\label{lem1}
There exists a consant $ \beta>0$ independent of the mesh-size $h$ such that 

\[ \sup_{\bv_h \in \bV_h} \frac{\int_{\Omega}\nabla_h \cdot \bv_h\,q_h\,d\bx}{ 
\|\bv_h\|_{1,h} } \geq \beta \|q_h\|_{0,\Omega},\quad q_h \in K_h.
\]
\end{lemma}
From the standard theory of saddle point problems this lemma implies the 
following lemma  \cite{BF91,Bra01}. 
\begin{lemma}\label{lem2}
Since the two spaces $\bH^1_0(\Omega)$ and $L^2_0(\Omega)$ satisfy the 
inf-sup condition 
\[ \sup_{\bu \in \bH^1_0(\Omega)}
\frac{\int_{\Omega}\nabla \cdot \bu \, \mu\,d\bx}{\|u\|_{1,\Omega}, } 
\geq \beta \|\mu\|_{L^2(\Omega)},\quad \mu \in L^2_0(\Omega), \]
where 
\[ L_0^2(\Omega) = \left\{v \in L^2(\Omega):\, \int_{\Omega} v\,d\bx=0\right\},\]
there exists a bounded linear projector  $\bPi_h : \bH^1_0(\Omega) \rightarrow  [W_h]^2$ 
such that  \[\int_{\Omega} \nabla_h \cdot \bPi_h \bu \,\mu_h \, d\bx = 
\int_{\Omega} \nabla \cdot  \bu\, \mu_h \, d\bx,\quad \mu_h \in  K_h \cap L^2_0(\Omega).\]
The boundedness means that there exists a constant $C$ independent of $h$ such that 
\[ \|\bPi_h\bu \|_{1,h} \leq C \|\bu\|_{1,\Omega}.\]
\end{lemma}
Let $\Pi_h :H^1_0(\Omega) \rightarrow W_h$ be the scalar version of the 
projector $\bPi_h$. Then this projector $\Pi_h$ is bounded and has the following property 
\cite{Lam13c}.

\begin{lemma}\label{lem3}
Let $ v \in H^1_0(\Omega)$.
The interpolation  operator $\Pi_h$  satisfies
\[ \int_{\Omega} \nabla_h \Pi_h v\cdot \boeta_h\, d\bx  = 
\int_{\Omega} \nabla v\cdot \boeta_h\, d\bx,\;\boeta_h \in \bM_h.\]
\end{lemma}

Now we present the proof of Theorem \ref{th1}. 
\begin{proof}
Let $\bQ_h:\bV \rightarrow \bV_h$ be the orthogonal 
projection. Then we have 
\[ \int_{\Omega} \bQ_h \bv\cdot \bw_h\,d\bx = 
\int_{\Omega}   \bv\cdot \bw_h \,d\bx,\quad \bw_h \in \bM_h.\]
Moreover, from Lemma \ref{lem3} we 
have 
\[  \int_{\Omega} \nabla_h \Pi_h v\cdot \boeta_h\, d\bx  = 
\int_{\Omega} \nabla v\cdot \boeta_h\, d\bx,\;\boeta_h \in \bM_h.\]
Hence we have \eqref{op} :
\[ b( \bQ_h\bpsi,\Pi_h v;\boeta_h) =  b(\bpsi,v;\boeta_h),\quad 
\boeta_h \in \bM_h,\]
since 
\[ b( \bQ_h\bpsi,\Pi_h v;\boeta_h) = 
\int_{\Omega} \left(\bQ_h\bpsi-\nabla_h\Pi_hv\right)\cdot \boeta_h\,d\bx= 
\int_{\Omega} \left(\bpsi-\nabla v\right)\cdot \boeta_h\,d\bx.
\]
Thus the theorem is proved.
\end{proof}

\subsection{Second example for $\bM_h$}
Our second example of the discrete Lagrange 
space is based on a biorthogonal system \cite{BWHabil,KLP01,Lam06}.
Let $N$ be the number of vertices in the finite element mesh, and 
$\{\varphi_1,\ldots,\varphi_N\}$ be the finite element basis of 
$K_h$. We construct a space 
$M_h$ spanned by the basis 
$\{\xi_1,\ldots,\xi_N\}$, where  the basis functions of $K_h$ and 
$M_h$ satisfy a condition of biorthogonality relation
\begin{eqnarray} \label{biorth}
  \int_{\Omega} \xi_i \ \varphi_j \, d\bx = c_j \delta_{ij}\, ,
\quad c_j\neq 0,\; 1\le i,j \le N\, ,
\end{eqnarray}
 where  $\delta_{ij}$ is 
 the Kronecker symbol, and $c_j$ a scaling factor. 
 The scaling factor can be chosen so that 
 $\int_{T} \xi_i\,d\bx=\int_{T} \varphi_i\,d\bx$.

 The basis functions of $M_h$ are also associated with 
 the vertices of the finite element mesh $\CT_h$, and 
 they are  constructed 
locally on a reference element $\hat T$. 
The local basis functions of $M_h$ on the 
reference triangle $\hat T:=\{(x,y) :\, 0\leq x,0\leq y,x+y\leq 1\}$ are given by 
\begin{eqnarray*}
  \hat \xi_1:=3-4x-4y,\,
  \hat\xi_2:=4x-1,\;\text{and}\;
  \hat\xi_3:=4y-1,
\end{eqnarray*}
associated with its three vertices $(0,0)$, $(1,0)$ and $(0,1)$, 
respectively. 
Note that the sum of all basis functions is one. 

The global basis functions for the space $M_h$ are constructed by 
glueing the local basis functions together in the same way 
as the global basis functions of $K_h$ are constructed. 
Such a biorthogonal system is first used in the context of 
mortar finite elements \cite{BWHabil,KLP01,Lam06}. Construction of basis functions 
of $M_h$ satisfying the biorthogonality and an optimal approximation
 property for a higher order finite element space is considered 
in \cite{Lam06}. 
We now prove Theorem \ref{th1} using the discrete Lagrange multiplier space 
$\bM_h := [M_h]^2$.  First we prove the 
following lemma. 
\begin{lemma}\label{lem4}
There exists a consant $ \beta>0$ independent of the mesh-size $h$ such that 
\[ \sup_{\bv_h \in \bV_h} \frac{\int_{\Omega}\nabla_h \cdot \bv_h\,q_h\,d\bx}{ 
\|\bv_h\|_{1,h} } \geq \beta \|q_h\|_{0,\Omega},\quad q_h \in M_h.
\]
\end{lemma}
\begin{proof}
We consider an operator $I_h : M_h \rightarrow K_h$ such that 
\[ I_h \mu_h = \sum_{i=1}^Nc_i\varphi_i \quad \text{for}\; 
\mu_h = \sum_{i=1}^N c_i \xi_i.\]
We first note that the basis functions of $M_h$ are constructed in such a way that 
\[ \int_{T} \xi_i \,d\bx =  \int_{T} \varphi_i \,d\bx,\quad 1 \leq i \leq N.
\] 
Let $\mu_h \in M_h$, and $I_h \mu_h \in K_h$. Then 
\[ \sup_{\bv_h \in \bV_h} \frac{\int_{\Omega}\nabla_h \cdot \bv_h\,\mu_h\,d\bx } { 
\|\bv_h\|_{1,h} } = \sup_{\bv_h \in \bV_h} \frac{\int_{\Omega}\nabla_h \cdot \bv_h\,I_h \mu_h\,d\bx } { 
\|\bv_h\|_{1,h} } \geq \beta \|I_h\mu_h\|_{0,\Omega}.
\]
The result follows by using the fact that 
$\|I_h\mu_h\|^2_0$, $\|\mu_h\|^2_0$ and 
$\sum_{i=1}^Nc^2_i h^2_i$ are equivalent, where $h_i$ is the local 
mesh-size at the $i$th node of $\CT_h$.
\end{proof}
We can apply the standard theory of saddle point problems as 
in  Lemma \ref{lem2} to get the following result.
\begin{lemma}\label{lem5}
There exists a bounded linear projector  $\bPi_h : \bH^1_0(\Omega) \rightarrow  [W_h]^2$ 
such that  \[\int_{\Omega} \nabla_h \cdot \bPi_h \bu \,\mu_h \, d\bx = 
\int_{\Omega} \nabla \cdot  \bu\, \mu_h \, d\bx,\quad \mu_h \in  K_h \cap L^2_0(\Omega).\]
\end{lemma}
Thus we have the existence of a bounded projector $\Pi_h :H^1_0(\Omega) \rightarrow W_h$ 
as in our first example as the scalar version of $\bPi_h$.
\begin{lemma}\label{lem6}
Let $ v \in H^1_0(\Omega)$.
The interpolation  operator $\Pi_h$  satisfies
\[ \int_{\Omega} \nabla_h \Pi_h v\cdot \boeta_h\, d\bx  = 
\int_{\Omega} \nabla v\cdot \boeta_h\, d\bx,\;\boeta_h \in \bM_h.\]
\end{lemma}

Now we present the proof of Theorem \ref{th1} for the second 
example of the discrete Lagrange multiplier space.
\begin{proof}
Let $\bQ_h:\bV \rightarrow \bV_h$ be a quasi-projection defined as 
\begin{equation}\label{quasi}
\int_{\Omega} \bQ_h \bv\cdot \bw_h\,d\bx = 
\int_{\Omega}   \bv\cdot \bw_h \,d\bx,\quad \bw_h \in \bM_h.
\end{equation}
This quasi-projection is well-defined due to Assumptions 
\ref{A1A2}(i)--(ii). We note that for the projection $\Pi_h$ from Lemma \ref{lem6} 
we have
\[ b( \bQ_h\bpsi,\Pi_h v;\boeta_h) = 
\int_{\Omega} \left(\bQ_h\bpsi-\nabla_h\Pi_hv\right)\cdot \boeta_h\,d\bx= 
\int_{\Omega} \left(\bpsi-\nabla v\right)\cdot \boeta_h\,d\bx,
\]
which yields  \eqref{op}: 
\[ b( \bQ_h\bpsi,\Pi_h v;\boeta_h) =  b(\bpsi,v;\boeta_h),\quad 
\boeta_h \in \bM_h.\]

Hence the theorem is proved.
\end{proof}

\section{Clamped boundary condition} \label{sec:clbc}
It is more difficult to construct a discrete Lagrange multiplier space 
for the {\em clamped} boundary condition. 
For example, for the discrete Lagrange multiplier space in the first example above if we choose 
$\bM_h=[K_h]^2$  Assumption \ref{A1A2}(i) is 
violated, and if we choose $\bM_h=[K^0_h]^2$, 
Assumption \ref{A1A2}(iii) is violated  leading to 
a sub-optimal approximation property. 
In order to satisfy Assumption 1(iii), the discrete Lagrange multiplier space
should contain constants in $\Omega$, and this does not happen 
if we choose $\bM_h =[K_h^0]^2$. While 
it may be possible to work around without Assumption \ref{A1A2}(i) 
we see two difficulties if we remove this assumption. 
The first difficulty is that the analysis will be much more 
difficult. The second difficulty is that the Gram matrix 
between the basis functions of $\bM_h$ and $\bV_h$ will not be 
a square matrix. If the Gram matrix is not square, 
we cannot statically condense out the 
degrees of freedom corresponding the Lagrange multiplier 
from the algebraic system. 

We now propose a modification of the discrete Lagrange multiplier space 
 \cite{Lam13c} to adapt to the situation of {\em clamped} boundary condition,
 which combines the idea of mortar finite element techniques \cite{BD98}
with that of \cite{AB93} to satisfy Assumptions \ref{A1A2}(i) 
and \ref{A1A2}(iii).  The presented modification 
is exactly as in \cite{Lam13c}. We repeat the approach 
here for completeness.

We start with splitting the basis functions of $M_h$ to two groups: 
basis functions associated with the inner vertices of $\CT_h$ and 
basis functions associated with the boundary vertices in $\CT_h$. 
Let $L \in \bN$ with $L<N$ be the number of inner vertices in $\CT_h$. 
Let \[ \CB^1_h= \{\varphi_1,\varphi_2,\cdots,\varphi_L\}, \quad\text{and}\quad 
\CB^2_h = \{\varphi_{L+1},\cdots, \varphi_N\}\] 
be the two sets of basis functions  of $K_h$ associated with the 
inner and boundary vertices in $\CT_h$, respectively. 
In the following, we assume that each triangle has at least 
one interior vertex. A necessary modification 
for the case where a triangle has all its vertices on the boundary 
is given in \cite{BD98}. Let $\CN$, $\CN_0$ and $\partial\CN$ be the set of all vertices 
of $\CT_h$, the vertices of $\CT_h$ interior to $\Omega$, and 
the vertices of $\CT_h$ on the boundary of $\Omega$, 
respectively.  We define the set of all 
vertices which share a common edge with the vertex  $ i \in\CN$ as 
\[ \cS_i = \{j: \text{$i$ and $j$ share a common edge}\},\]
and the set of neighbouring vertices of $i \in \CN_0$ as 
\[ \cI_i =\{ j \in \CN_0:\,j \in \cS_i\}.\]

Then the set of all those interior vertices which have a neighbour on the 
boundary of $\Omega$ is defined as 
\[ \cI = \bigcup_{i \in \partial\CN}\cI_i,\qquad \text{See Figure \ref{SI}}.
\] 
 \begin{figure}[!ht]
\begin{center}
 \includegraphics[width =0.68\textwidth]{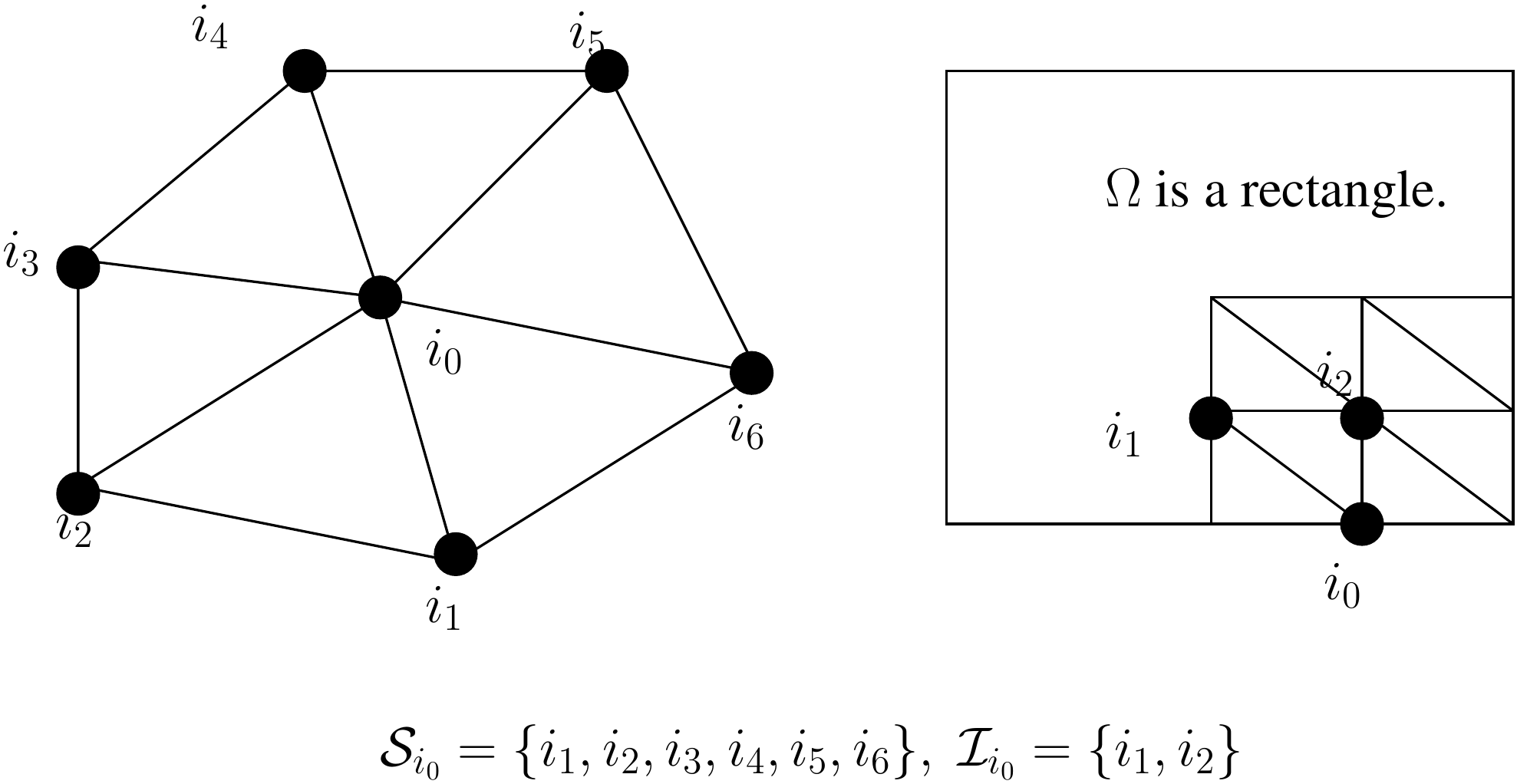}
 \caption{Examples for ${\mathcal S}_i$ and ${\mathcal I}_i$} 
\end{center}
 \label{SI}
\end{figure}

The finite element basis functions $\{\phi_1,\phi_2,\cdots,\phi_L\}$ 
for $\widetilde M_h$ are defined as 
\[ \phi_i = \begin{cases} \varphi_i, & i \in \CN_0\backslash \cI\\
 \varphi_i + \sum_{j \in \partial \CN \cap \cS_i}A_{j,i} \varphi_j,\; 
A_{j,i} \geq 0, & i \in \cI
\end{cases}.
\]
We can immediately see that $\dim \bM_h = \dim \bV_h$. Moreover, 
if the coefficients $A_{i,j}$ are  chosen to satisfy 
\[ \sum_{j \in \cS_i} A_{i,j} =1,\quad i \in \cI,\]
Assumptions 1(ii) and 1(iii) are also satisfied,
see  \cite{BD98} for a proof. 
The vector Lagrange multiplier space is defined as $\widetilde \bM_h =[\widetilde M_h]^2$.  
Since $\widetilde \bM_h \subset \bM_h$ for both examples
 we have the following theorem.
\begin{theorem}\label{th3}
There exist two bounded linear projectors  
$\bQ_h :\bH_0^1(\Omega) \rightarrow \bV_h$ 
and 
$\Pi_h : H_0^1(\Omega) \rightarrow W_h$ 
for which 
\[ b( Q_h\bpsi,\Pi_h v;\boeta_h) = b(\bpsi,v;\boeta_h),\quad 
\boeta_h \in \widetilde \bM_h.\]
\end{theorem}

\section{Positive-definite formulation}\label{sec:pd}
Now we give a positive definition formulation for our 
finite element scheme. Let 
$\bR_h:\bL^2(\Omega) \rightarrow \bM_h$ be the orthogonal projection 
defined as 
\[ \int_{\Omega} \bR_h \bv \cdot \boeta_h\,d\bx = \int_{\Omega} \bv \cdot \boeta_h\,d\bx,\quad 
\bv \in \bL^2(\Omega)\]
for both examples of the discrete Lagrange multiplier 
space and both types of boundary condition. Then the second equation 
of the discrete saddle point problem \eqref{dsaddle} can be written as 
\[ \bzeta_h = \frac{\lambda(1-t^2)}{t^2} \bR_h (\bphi_h-\nabla_h u_h).\]
Using this result in the first equation of  \eqref{dsaddle} 
the positive-definite formulation is to find 
$(\bphi_h,u_h) \in \bV_h \times W_h$   such that 
\begin{eqnarray*}
A(\bphi_h,u_h;\bpsi_h,v_h) = \ell(v_h),\quad 
(\bpsi_h,v_h) \in \bV_h \times W_h,
\end{eqnarray*}
where 
\begin{eqnarray*}
A(\bphi_h,u_h;\bpsi_h,v_h)  = 
\int_{\Omega} \CC \bepsilon(\bphi_h):\bepsilon(\bpsi_h)\,d\bx +
\lambda \int_{\Omega} (\bphi_h-\nabla u_h)\cdot(\bpsi_h-\nabla_h v_h)\,d\bx+\\
\frac{\lambda(1-t^2)}{t^2} 
 \int_{\Omega} (\bpsi_h-\nabla_h v_h)\cdot\bR_h (\bphi_h-\nabla_h u_h) \,d\bx.
\end{eqnarray*}
The disadvantage of this positive-definite system is that the action of $\bR_h$ cannot be 
efficiently computed. Now we present another way of 
getting a positive-definite form for the second example of the 
discrete Lagrange multiplier space, 
which can be efficiently computed. 
Note that  the two sets of  basis functions of $\bM_h$ and $\bV_h$ form 
a biorthogonal system for the second example.
 Then the Gram matrix $\tD$ associated with 
these two sets of basis functions will be diagonal. Then 
putting $v_h=0$ in the first equation of the saddle point 
system \eqref{dsaddle}, we have 
\begin{eqnarray*}
 \int_{\Omega} \CC \bepsilon(\bphi_h):\bepsilon(\bpsi_h)\,d\bx +
\lambda \int_{\Omega} (\bphi_h-\nabla_h u_h)\cdot \bpsi_h\,d\bx + 
\int_{\Omega} \bpsi_h\cdot\bzeta_h\,d\bx =0 .
\end{eqnarray*}
Note that the Gram matrix $\tD$ is associated with 
the inner product $\int_{\Omega} \bpsi_h\cdot\bzeta_h\,d\bx $.
Thus with  suitable choices of matrices $\tA$ and $\tB$, 
the algebraic form of this equation becomes 
\[ \tA \bphi_h + \tB^T u_h + \tD \bzeta_h = 0.\]
This equation can be solved for $\bzeta_h$ as 
\[ \bzeta_h = - \tD^{-1} \left(\tA \bphi_h + \tB^T u_h \right).\]
 Thus we can statically condense out 
the Lagrange multiplier $\bzeta_h$ from the saddle point system. 
This leads to a reduced and positive definite system. Hence  an 
efficient solution technique can be applied to 
solve the arising linear system.

\section{Conclusion}
We have presented a new finite element method 
for Reissner-Mindlin plate equations using 
nonconforming Crouzeix-Raviart finite element basis functions for 
the transverse displacement, and the standard linear finite element 
 for the rotation of the transverse normal vector.
 We have also shown  two examples of the discrete 
 Lagrange multiplier space for the presented 
 finite element approach. We note that the second 
 example for the discrete Lagrange multiplier space 
 provides a more efficient numerical method as 
 the degrees of freedom corresponding to 
 the Lagrange multiplier can be statically condensed out 
 from the system in this case just by inverting a diagonal matrix. 
 
\bibliographystyle{alpha}
\bibliography{total}
\end{document}